\documentclass[a4paper, final]{amsart}

\usepackage[utf8]{inputenc}
\usepackage[T1]{fontenc}
\usepackage[english]{babel}
\usepackage{amsmath,amsthm}
\usepackage{ae}
\usepackage{icomma}
\usepackage{units}
\usepackage{color}
\usepackage{graphicx}
\usepackage{bbm}
\usepackage{caption}
\usepackage{array}
\usepackage[hmarginratio=1:1]{geometry}
\usepackage[hyphens]{url}
\usepackage[pdfpagelabels=false]{hyperref}
\usepackage{mathrsfs}
\usepackage{amssymb}
\usepackage{placeins} 
\usepackage{tikz}
\usepackage{float} 
\usepackage[nodayofweek]{datetime}
\usepackage{enumitem}
\usepackage{mathtools}
\usepackage[numbers, sort]{natbib}
\usepackage{comment}
\mathtoolsset{showonlyrefs=true}

\newcommand{\R}{\ensuremath{\mathbb{R}}}
\newcommand{\N}{\ensuremath{\mathbb{N}}}
\DeclareMathOperator{\dist}{\textnormal{dist}}
\DeclareMathOperator{\Tr}{Tr}
\renewcommand{\S}{\ensuremath{\mathbb{S}}}
\newcommand{\LT}[2]{L_{#1, #2}^{\textrm{cl}}}
\newcommand{\eps}{\ensuremath{\varepsilon}}
\newcommand{\grad}{\ensuremath{\nabla}}
\newcommand{\ptS}{\hspace{-1pt} }

\newtheorem{theorem}{Theorem}[section]

\newtheorem{corollary}[theorem]{Corollary}

\numberwithin{theorem}{section}
\numberwithin{definition}{section}

\newcommand{\thmref}[1]{Theorem~\ref{#1}}

\newlength{\MyLen}	
\newcolumntype{C}[1]{>{\centering\arraybackslash}p{#1}}
\newcolumntype{M}[1]{>{$}C{#1}<{$}}
\newcolumntype{L}{>{$}l<{$}}
\newcolumntype{R}{>{$}r<{$}}

\begin{document}

\pagenumbering{arabic}

\title{On the remainder term of the Berezin inequality on a convex domain}

\author[S. Larson]{Simon LARSON}
\address{Department of Mathematics, KTH Royal Institute of Technology, SE-100 44 Stockholm, Sweden}
\email{simla@math.kth.se}


\begin{abstract}
	We study the Dirichlet eigenvalues of the Laplacian on a convex domain in $\mathbb{R}^n$, with $n\geq 2$. In particular, we generalize and improve upper bounds for the Riesz means of order $\sigma\geq 3/2$ established in an article by Geisinger, Laptev and Weidl. This is achieved by refining estimates for a negative second term in the Berezin inequality. The obtained remainder term reflects the correct order of growth in the semi-classical limit and depends only on the measure of the boundary of the domain. We emphasize that such an improvement is for general $\Omega\subset\R^n$ not possible and was previously known to hold only for planar convex domains satisfying certain geometric conditions.

	As a corollary we obtain lower bounds for the individual eigenvalues $\lambda_k$, which for a certain range of $k$ improves the Li--Yau inequality for convex domains. However, for convex domains one can using different methods obtain even stronger lower bounds for $\lambda_k$.
\end{abstract}

\maketitle


\section{Introduction}

Let $\Omega$ be an open subset of $\R^n$ and let $-\Delta_\Omega$ be the Dirichlet Laplace operator on $L^2(\Omega)$, defined in the quadratic form sense with form domain $H_0^1(\Omega)$. If the volume of $\Omega$, which we denote by $|\Omega|$, is finite then the embedding $H_0^1(\Omega) \hookrightarrow L^2(\Omega)$ is compact and the spectrum of $-\Delta_\Omega$ is discrete. Further, the spectrum is positive and accumulates only at infinity. Thus we can write it as an increasing sequence of eigenvalues: 
\begin{equation}
	0< \lambda_1(\Omega)\leq \lambda_2(\Omega) \leq \lambda_3(\Omega) \leq \ldots,
\end{equation}
where an eigenvalue is repeated according to its multiplicity.

Letting $x_{\pm}=(|x|\pm x)/2$, the Riesz means of these eigenvalues are defined, for $\Lambda>0$, by
\begin{equation}
	\sum_{k=1}^\infty (\Lambda-\lambda_k)_+^\sigma = \Tr (-\Delta_\Omega-\Lambda)^\sigma_-, \quad \sigma\geq 0.
\end{equation}
In what follows we will be interested in establishing upper bounds for these means. In particular, we will study the case $\sigma\geq 3/2$ when $\Omega$ is convex.

The classical Weyl asymptotic formula (see \cite{WeylAsymptotic}) states that for $\Omega\subset \R^n$ and $\sigma\geq 0$ the identity
\begin{equation}\label{eq:WeylAsymptotic}
 	\Tr(-\Delta_\Omega-\Lambda)_-^\sigma = \LT{\sigma}{n}|\Omega| \Lambda^{\sigma+n/2}+o(\Lambda^{\sigma+n/2})
\end{equation}
holds as $\Lambda\to \infty$. Here, and in what follows, $\LT{\sigma}{n}$ denotes the Lieb--Thirring constant:
\begin{equation}\label{eq:LiebThirringConstant}
	\LT{\sigma}{n} = \frac{\Gamma(\sigma+1)}{(4\pi)^{n/2}\Gamma(\sigma+1+n/2)}.
\end{equation}

Following the work of Weyl the second term of the asymptotics has been further studied (see, for instance, \cite{zbMATH02596129,MR2885166,MR0609014,MR575202,Ivrii,MR1414899}). Under certain conditions on the set $\Omega$ and its boundary $\partial\Omega$ it was proved by Ivrii in~\cite{MR575202} that
\begin{align}\label{eq:IvriiSecondterm}
	\Tr(-\Delta_\Omega-\Lambda)_-^\sigma = \LT{\sigma}{n} |\Omega|\Lambda^{\sigma+n/2}-\frac{1}{4}\LT{\sigma}{n-1}|\partial \Omega| \Lambda^{\sigma+(n-1)/2} + o(\Lambda^{\sigma+(n-1)/2})
\end{align}
holds as $\Lambda\to \infty$, for $\sigma \geq 1$ this was later generalized to a larger class of domains by Frank and Geisinger~\cite{MR2885166}. To simplify notation we write $|\Omega|$ for the $n$-dimensional volume of $\Omega$, and $|\partial\Omega|$ for the $(n-1)$-dimensional surface area of its boundary. 

In \cite{Berezin} Berezin proved that for $\Omega\subset \R^n$ and $\Lambda>0$ the convex Riesz eigenvalue means, that is when~$\sigma\geq 1$, satisfy the bound
\begin{equation}\label{eq:BerezinInequality}
 	\Tr(-\Delta_\Omega-\Lambda)_-^\sigma \leq \frac{1}{(2\pi)^{n}} \iint_{\Omega\times \R^n} (|p|^2-\Lambda)^\sigma_- \, dp\, dx = \LT{\sigma}{n} |\Omega| \Lambda^{\sigma+n/2}.
\end{equation} 
From the Weyl asymptotics \eqref{eq:WeylAsymptotic} it follows that the constant $\LT{\sigma}{n}$ in this bound is sharp. That~\eqref{eq:BerezinInequality} remains true also for $\sigma=0$ coincides with the P\'olya conjecture on the number of eigenvalues of $-\Delta_\Omega$ less than $\Lambda$ (see~\cite{Polya}). In view of~\eqref{eq:IvriiSecondterm} this raises the question of whether one, in a similar manner as for the semi-classical limit, can improve Berezin's inequality by a negative remainder term. 

Given an open set $\Omega\subset\R^n$ one can increase $|\partial \Omega|$ without significantly increasing $\lambda_k(\Omega)$. Thus, it is in general, for any $C>0$, not possible to subtract a term $C |\partial\Omega|\Lambda^{\sigma+(n-1)/2}$ from the right-hand side of~\eqref{eq:BerezinInequality}.
However, the main result of this paper is that if we restrict our attention to convex sets and $\sigma\geq 3/2$ such an improvement is possible. This result is contained in the following theorem which generalizes a result obtained by Geisinger, Laptev and Weidl~\cite[Theorem~5.1]{LaptevGeisingerWeidl} for convex sets in $\R^2$ satisfying certain geometric assumptions (see~\thmref{thm:GLWtheorem5.1}).
\begin{theorem}\label{thm:ImprovedTracebound}
  Let $\Omega \subset \R^n$ be a bounded, convex domain with inradius $r$ and let $\sigma \geq 3/2$. Then there exists a constant $C(\sigma, n)>0$ such that
  \begin{equation}
    \Tr(-\Delta_\Omega-\Lambda)^\sigma_- = 0  \hspace{218pt} \textrm{if } \Lambda \leq \frac{\pi^2}{4 r^2}\phantom{.}
  \end{equation}
  and
  \begin{equation}
    \Tr(-\Delta_\Omega-\Lambda)^\sigma_- \leq \LT{\sigma}{n} |\Omega| \Lambda^{\sigma+n/2} - C(\sigma, n) \LT{\sigma}{n-1} |\partial \Omega| \Lambda^{\sigma+(n-1)/2} \qquad \textrm{if } \Lambda > \frac{\pi^2}{4 r^2}.
  \end{equation}
  Further, we provide upper and lower bounds for the constants $C(\sigma, n)$.
\end{theorem}

Using techniques from~\cite{Laptev2} this result can be applied to find improved bounds for Riesz means on product domains, $\Omega=\Omega_1\!\times \Omega_2$, where one of the factors is a convex domain. These considerations lead to the following corollary:

\begin{corollary}\label{cor:Product_Domains}
	Let $\Omega=\Omega_1\! \times \Omega_2$, where $\Omega_1\subset \R^{n_1}$ is a bounded, convex domain with inradius $r$ and\/ $\Omega_2\subset \R^{n_2}$ is bounded and open. Assume that $\sigma+ n_2/2 \geq 3/2$ and that for all $\Lambda>0$
	\begin{equation}\Tr(-\Delta_{\Omega_2}-\Lambda)^\sigma_- \leq \LT{\sigma}{n_2}|\Omega_2|\Lambda^{\sigma+n_2/2}.
\end{equation}
With $n=n_1+n_2$ we for $\Lambda \leq \frac{\pi^2}{4r^2}$ have that
\begin{equation}
	\Tr(-\Delta_\Omega-\Lambda)^\sigma_- =0, \hspace{259pt}
\end{equation}
and if $\Lambda> \frac{\pi^2}{4r^2}$ then
	\begin{equation}
		\Tr(-\Delta_\Omega-\Lambda)^\sigma_- \leq \LT{\sigma}{n}|\Omega|\Lambda^{\sigma+n/2}-C(n_1, \sigma+n_2/2) \LT{\sigma}{n-1}|\Omega_2|\,|\partial\Omega_1| \Lambda^{\sigma+(n-1)/2},
	\end{equation}
where $C(\sigma, n)$ is the constant appearing in Theorem~\ref{thm:ImprovedTracebound}.
\end{corollary}

In particular, if $n_2\geq 3$ and $\Omega_2$ satisfies the P\'olya conjecture, for instance if $\Omega_2$ is a tiling domain, we may apply this with $\sigma=0$. Thus we obtain examples of domains for which the P\'olya conjecture is true even if we subtract, from the right-hand side of~\eqref{eq:BerezinInequality}, a term of order $\Lambda^{(n-1)/2}$.

\begin{proof}[Proof of Corollary~\ref{cor:Product_Domains}]
Since $\Omega=\Omega_1\!\times \Omega_2$ we have that the eigenvalues of $-\Delta_\Omega$ are given by
\begin{equation}
	\lambda_{kl}= \eta_k+\nu_l,
\end{equation}
where $\eta_k$ and $\nu_l$ are the eigenvalues of $-\Delta_{\Omega_1}$ and $-\Delta_{\Omega_2}$, respectively. Thus we find that
\begin{align}
	\Tr(-\Delta_\Omega-\Lambda)^\sigma_-
	=
	\sum_{\lambda_{kl}\leq \Lambda} (\Lambda-\lambda_{kl})^\sigma
	=
	\sum_{\eta_k \leq \Lambda}\Bigl(\sum_{\nu_l\leq \Lambda-\eta_k}((\Lambda-\eta_k)-\nu_l)^\sigma\Bigr).
\end{align}
By the assumptions on $\Omega_2$, one obtains that
\begin{align}
	\Tr(-\Delta_{\Omega}-\Lambda)^\sigma_-
	\leq
	\LT{\sigma}{n_2}|\Omega_2|\sum_{\eta_k\leq \Lambda}(\Lambda-\eta_k)^{\sigma+n_2/2} = \LT{\sigma}{n_2}|\Omega_2|\Tr(-\Delta_{\Omega_1}-\Lambda)^{\sigma+n_2/2}_-.
\end{align}
Applying Theorem~\ref{thm:ImprovedTracebound} and using that $\LT{\sigma}{n_2}\LT{\sigma+n_2/2}{n_1}\!=\LT{\sigma}{n_1+n_2}$ yields the result.
\end{proof}

As it stands in~\cite{LaptevGeisingerWeidl} the theorem corresponding to Theorem~\ref{thm:ImprovedTracebound} above contains an error. This error appears when the Aizenman--Lieb argument (see~\cite{AizenmanLieb}) is used together with a bound for the case $\sigma=3/2$ to obtain bounds for larger values of $\sigma$. However, the proof that is used in~\cite{LaptevGeisingerWeidl} for the case of $\sigma=3/2$ generalizes without any difficulty to arbitrary $\sigma\geq 3/2$ (this is the method we use here). The only difference is that instead of a constant depending only on the dimension we obtain one depending also on the parameter $\sigma$, namely $C(\sigma, n)$. In fact, it is not very difficult to prove that this constant must depend on both $\sigma$ and $n$. 

The first result in the direction of improving Berezin's inequality~\eqref{eq:BerezinInequality} is due to Melas, who in~\cite{MR1933356} obtains an improvement for all $\sigma\geq 1$. However, the negative correction term that was established in \cite{MR1933356} is not of the same order in $\Lambda$ as the correction term in the semi-classical asymptotics~\eqref{eq:IvriiSecondterm}. In the two-dimensional case it was proved in~\cite{MR2486669} that the order of the remainder term can be chosen arbitrarily close to the asymptotically correct one, namely~$\sigma+1/2$.

In the case of $\sigma\geq 3/2$, which is the case studied here, it was established in~\cite{MR2509788} that the Berezin inequality, for open sets $\Omega\subset \R^n$, can be strengthened by a negative term of the same order in $\Lambda$ as the second term in~\eqref{eq:IvriiSecondterm}. However, as remarked earlier any uniform improvement of~\eqref{eq:BerezinInequality} must depend on other geometric quantities. For instance, the remainder term found in~\cite{MR2509788} depends on projections onto hyperplanes and in~\cite{MR2725046} the authors derive a remainder term, of the correct order, depending only on $|\Omega|$. 

The approach of~\cite{MR2509788} relies on using 
Lieb--Thirring inequalities for Schr\"odinger operators with operator-valued potentials, see \cite{MR1756570}, and reducing the problem to trace estimates for the one-dimensional Laplacian on open intervals. In \cite{LaptevGeisingerWeidl} the authors employ the same approach but with different estimates for the one-dimensional problem. Moreover, the authors of~\cite{LaptevGeisingerWeidl} are able to refine these estimates if $\Omega$ is convex. We summarize these refinements in the following theorems.

\begin{theorem}[\cite{LaptevGeisingerWeidl}, Corollary~3.5]\label{thm:GLWCorollary3.5}
	Let $\Omega\subset \R^n$ be a bounded, convex domain with smooth boundary and assume that at each point the principal curvatures of $\partial \Omega$ are bounded from above by $1/K$. Then, for $\sigma\geq 3/2$ and all $\Lambda>0$ we have that
	\begin{equation}\label{eq:GLWcor_3.5}
		\Tr(-\Delta_\Omega-\Lambda)_-^\sigma \leq \LT{\sigma}{n} |\Omega| \Lambda^{\sigma+n/2} -
		\LT{\sigma}{n}2^{-n-2}|\partial \Omega| \Lambda^{\sigma+(n-1)/2}\int_0^1 \Bigl(1- \frac{n-1}{4 K\sqrt{\Lambda}}\, t\Bigr)_{\!+}\, dt.
	\end{equation}
\end{theorem}

\begin{theorem}[\cite{LaptevGeisingerWeidl}, Theorem~5.1]\label{thm:GLWtheorem5.1}
	Let $\Omega\subset \R^2$ be a bounded, convex domain with width $w$ and let\/ $\Omega_t=\{x\in \Omega : \dist(x, \Omega^c)\geq t\}$ denote its inner parallel set at distance $t\geq 0$. Further, assume that each\/ $\Omega_t$\ptS satisfies the estimate
  \begin{equation}\label{eq:ConjecturedIneqGLW}
		|\partial \Omega_t| \geq \Bigl(1-\frac{3 t}{w}\Bigr)_{\!+} |\partial\Omega|.
	\end{equation}
	Then, for $\sigma\geq 3/2$ we have that
	\begin{equation}
	 	\Tr (-\Delta_\Omega - \Lambda )_-^\sigma \, = \, 0 \hspace{178pt} \textnormal{if } \Lambda \leq \frac{\pi^2}{w^2}\phantom{,}
	\end{equation} 
	and
	\begin{equation}
		\Tr ( -\Delta_\Omega - \Lambda)_-^\sigma \, \leq \, L^{cl}_{\sigma,2} \, |\Omega| \, \Lambda^{\sigma+1} - C(\sigma) \, L^{cl}_{\sigma,1} \, |\partial \Omega| \, \Lambda^{\sigma+1/2} \qquad \textnormal{if } \Lambda > \frac{\pi^2}{w^2},
	\end{equation}
	for some $C(\sigma)>0$. In particular
	\begin{equation}
		C(3/2) \geq \frac{11}{9\pi^2}-\frac{3}{20\pi^4}-\frac{2}{5\pi^2}\log \Bigl(\frac{4\pi}{3}\Bigr)\, > \, 0.0642.
	\end{equation}
\end{theorem}

As pointed out earlier the last theorem is stated in~\cite{LaptevGeisingerWeidl} with a constant not depending on $\sigma$ (in place of $C(\sigma)$), as we shall see such a statement cannot hold. However, the proof provided in~\cite{LaptevGeisingerWeidl} for the case $\sigma=3/2$ holds and through a straightforward generalization this can be used to prove the statement for all $\sigma\geq 3/2$.

Note that in \thmref{thm:GLWtheorem5.1} the remainder term reflects the correct order of growth in the semi-classical limit and depends only on~$|\partial\Omega|$. As remarked above this is not possible in general. In this paper we use bounds for the perimeter of inner parallel sets, obtained in \cite{Larson}, to refine and generalize both~\thmref{thm:GLWCorollary3.5} and \thmref{thm:GLWtheorem5.1} to arbitrary convex domains and any dimension. 

We begin Section~\ref{sec:ImprovedBerezin} with a short introduction to the theory and notation that we will need from~\cite{LaptevGeisingerWeidl}. We then proceed by applying the results of~\cite{Larson} to refine the arguments leading to the improved Berezin bounds. The generalization of \thmref{thm:GLWCorollary3.5} is proved in the same manner as in~\cite{LaptevGeisingerWeidl}, the only difference being the application of results from~\cite{Larson} instead of a version of Steiner's inequality (see~\cite{vandenBerg}). Also the argument leading to \thmref{thm:ImprovedTracebound}, the generalized version of \thmref{thm:GLWtheorem5.1}, is an almost step by step generalization of the proof given in~\cite{LaptevGeisingerWeidl}. However, for general dimension the computations become slightly more complicated.

In Section~\ref{sec:IndividualEigenvalues} we use the obtained improvements of \eqref{eq:BerezinInequality} to prove (implicit) lower bounds for individual eigenvalues $\lambda_k(\Omega)$, where $\Omega$ is convex. We are able to show that for a rather surprising number of the lower eigenvalues these bounds are an improvement of the Li--Yau inequality~\cite{MR701919}:
\begin{equation}\label{eq:LiYau}
  \lambda_k(\Omega) \,\geq\, \Gamma\Bigl(\frac{n}{2}+1\Bigr)^{\!2/n}\frac{4 \pi n}{n+2} \Bigl(\ptS\frac{k}{|\Omega|}\ptS\Bigr)^{\!2/n}.
\end{equation}

We note that for convex $\Omega$ one can, through different methods, improve the bounds given by~\eqref{eq:LiYau}, see~\cite{742593}. Even though our results in a certain range of $k$ provide better bounds than~\eqref{eq:LiYau}, they fail to improve the results of~\cite{742593} in general.


\section{An improved Berezin inequality for convex domains}\label{sec:ImprovedBerezin}

We begin with a short introduction of the relevant notation used in \cite{LaptevGeisingerWeidl} and~\cite{Larson}. For an open set $\Omega\subset\R^n$ (which in our case will be a convex set) we let, for $x\in \Omega$ and $u \in \S^{n-1}$,
\begin{align}
 	\theta(x, u) &= \inf\{ t>0 : x+tu\notin \Omega \},\\
 	d(x, u) &= \inf \{\theta(x, u), \theta(x, -u)\}
\end{align} 
and
\begin{equation}
	l(x, u) = \theta(x, u)+ \theta(x, -u).
\end{equation}
For a convex non-empty set $\Omega\subset \R^n$ we let $h(\Omega, \cdot\,)$ denote the support function of $\Omega$, which is defined by
\begin{equation}
	h(\Omega, x) = \sup_{y \in \Omega} 
	\langle y, x\rangle, \quad x\in \R^n.
\end{equation}
For a detailed account on properties of the support function and convex geometry in general we refer to Schneider's excellent book~\cite{Schneider1}.

Letting $\delta(x)$ denote the distance from $x$ to the boundary of $\Omega$ we have that 
\[
	\delta(x)=\inf_{u\in \S^{n-1}} \theta(x, u).
\]
We define the inradius $r$ and width $w$ of a convex set $\Omega$ by
\begin{align}
	r = \sup_{x\in \Omega} \delta(x),\qquad w = \inf_{u\in \S^{n-1}} h(\Omega, u)+h(\Omega, -u).
\end{align}
For a convex set $\Omega\subset \R^n$ with width $w$ it holds (see for instance \cite{Schneider1}) that 
\begin{equation}
	w = \inf_{\vphantom{\int^t}u\in \S^{n-1}} \sup_{x\in \Omega} \ l(x, u).
\end{equation}
The quantity on the right-hand side is in~\cite{LaptevGeisingerWeidl}, for a general domain $\Omega$, denoted by $l_0$. 

As in \thmref{thm:GLWtheorem5.1} we let $\Omega_t$ denote the inner parallel body of a convex set $\Omega$ at distance $t\geq 0$, which is defined by
\begin{equation}
	\Omega_t = \{ x\in \Omega : \dist(x, \Omega^c)\geq t\}.
\end{equation}
The inradius of $\Omega$ can now alternatively be written as $r=\sup \{t\geq0 : \Omega_t \neq \emptyset\}$. 

In~\cite{Larson} the main result is a lower bound for the $(n-1)$-dimensional surface area of the perimeter of an inner parallel set, the result is stated below and will be of central importance in what follows.

\begin{theorem}[\cite{Larson}, Theorem~1.2]\label{thm:PerimeterBoundInnerParallel}
 	Let $\Omega\subset \R^n$ be a convex domain with inradius $r$. Then, for any inner parallel set $\Omega_t$, $t\geq 0$, we have that
 	\begin{equation}
 		|\partial\Omega_t| \geq \Bigl(1- \frac{t}{r}\Bigr)^{\!n-1}_{\!+} |\partial \Omega|.
 	\end{equation}
 	Further, equality holds for some $t\in(0, r)$ if and only if $\Omega$ is homothetic to its form body. If this is the case equality holds for all $t\geq 0$.
\end{theorem} 
For the precise definition of the form body of $\Omega$ we refer to~\cite{Schneider1}. Since the exact conditions for equality will be of little importance, we will not include the precise definition.

For a fixed $\eps>0$ let
\begin{equation}
	A_\eps (x) = \{a \in \R^n \setminus \overline\Omega : |x-a|< \delta(x)+\eps \}
\end{equation}
and for any $x\in \Omega$ let
\begin{equation}
	\rho(x) = \inf_{\vphantom{\int^A}\eps >0} \sup_{\; a\in A_\eps(x)} \frac{|B_{\delta(x)}(a) \setminus \overline\Omega|}{|B_1(0)|\, |x-a|^n},
\end{equation}
where $B_\delta(x)$ denotes a ball of radius $\delta$ centred at $x\in \R^n$.
For a convex domain $\Omega$ we have that $\rho(x)>1/2$ for all $x\in \Omega$ (see~\cite{LaptevGeisingerWeidl}).

As in \cite{LaptevGeisingerWeidl} we for $\Lambda >0$ set
\begin{equation}
	M_\Omega (\Lambda) = \int_{R_\Omega(\Lambda)}\rho(x) \,dx,
\end{equation}
where $R_\Omega(\Lambda)= \{ x\in \Omega : \delta(x) < 1/(4 \sqrt{\Lambda}) \}= \Omega\setminus \Omega_{1/(4\sqrt{\Lambda})}$.

The following theorem and its proof in~\cite{LaptevGeisingerWeidl} form the starting point for most of the remaining arguments of this paper.
\begin{theorem}[\cite{LaptevGeisingerWeidl}, Theorem~3.3]\label{thm:GLWTheorem3.3}
Let $\Omega\subset \R^n$ be an open set with finite volume and $\sigma\geq 3/2$. Then for all $\Lambda>0$ we have that 
\begin{equation}
	\Tr(-\Delta_\Omega - \Lambda)_{-}^\sigma \leq \LT{\sigma}{n} |\Omega| \Lambda^{\sigma+n/2} - \LT{\sigma}{n} 2^{-n+1} \Lambda^{\sigma+n/2} M_\Omega (\Lambda).
\end{equation}
\end{theorem}

Using Theorem~\ref{thm:PerimeterBoundInnerParallel} and the same argument that leads to Corollary~3.5 in~\cite{LaptevGeisingerWeidl}, we deduce the following bound.
\begin{corollary}\label{cor:IntegralRemainder}
 	Let $\Omega \subset \R^n$ be a bounded, convex domain with inradius $r$. Then for all $\sigma \geq3/2$ and all $\Lambda>0$ we have that
 	\begin{align}
 		\Tr(-\Delta_\Omega-\Lambda)_-^\sigma 
 		&\leq
 		\LT{\sigma}{n} |\Omega| \Lambda^{\sigma+n/2}\\
 		& \quad
 		- \LT{\sigma}{n} 2^{-n-2}|\partial \Omega| \Lambda^{\sigma+(n-1)/2} \int_0^1 \Bigl(1- \frac{s}{4r \sqrt{\Lambda} }\Bigr)^{\!n-1}_{\!+} \,ds.
 	\end{align}
\end{corollary}
\begin{proof}
	Consider the remainder term in Theorem~\ref{thm:GLWTheorem3.3}. Inserting into the definition of $M_\Omega(\Lambda)$ that $\rho(x)>1/2$ when $\Omega$ is convex we find that
	\begin{align}
		M_\Omega(\Lambda) \,=\, \int_{R_\Omega(\Lambda)}\rho(x)\,dx
		\,>\,
			\int_{R_\Omega(\Lambda)}\frac{1}{2}\, dx
		\,=\,
			\frac{1}{2}\int_0^{1/(4\sqrt{\Lambda})} |\partial \Omega_t|\, dt.
	\end{align}
	Applying Theorem~\ref{thm:PerimeterBoundInnerParallel} yields
	\begin{align}
			M_\Omega(\Lambda) \,>\, \frac{|\partial\Omega|}{2} \int_0^{1/(4\sqrt{\Lambda})} \Bigl(1-\frac{t}{r}\Bigr)_{\!+}^{\!n-1}\,dt \,=\,
	\frac{|\partial \Omega|}{8 \sqrt{\Lambda}} \int_0^1 \Bigl(1-\frac{s}{4 r\sqrt{\Lambda}}\Bigr)_{\!+}^{\!n-1}\,ds,
	\end{align}
	which proves the claim.
\end{proof}

Using the inequality $\lambda_1(\Omega)\geq \frac{\pi^2}{4r^2}$ (see~\cite{MR589137}) Corollary~\ref{cor:IntegralRemainder} actually implies that \thmref{thm:ImprovedTracebound} holds for some positive constant $C(\sigma, n)$. However, by applying more refined techniques we can prove \thmref{thm:ImprovedTracebound} with substantially larger values for $C(\sigma, n)$.

We now turn our attention to the main results of this paper, namely~\thmref{thm:ImprovedTracebound}. As noted earlier this generalizes a result obtained in~\cite{LaptevGeisingerWeidl}, in particular we are able to relax certain geometric constraints and generalize the result to dimensions $n\geq 2$. We emphasize that the remainder term in~\thmref{thm:ImprovedTracebound} reflects the behaviour of the second term of the semi-classical limit $\Lambda\to \infty$, see \eqref{eq:IvriiSecondterm}. It has the correct order in $\Lambda$ and depends only on the size of $\partial \Omega$. Since it is not possible to obtain a uniform remainder term of this form for a general domain $\Omega\subset\R^n$, it would be of interest to know under what geometric conditions such a bound holds.

For $n=2$ and $\sigma=3/2$ the constant can be estimated in a similar manner as in~\cite{LaptevGeisingerWeidl} with the slightly improved result
\begin{align}
	C(3/2, 2) > 0.0846 > \frac{11}{9 \pi^2}-\frac{3}{20\pi^4}-\frac{2}{5\pi^2}\ln\Bigl(\frac{4\pi}{3}\Bigr)\approx 0.0642,
\end{align}
where the constant on the right-hand side is the one found by Geisinger, Laptev and Weidl. The lower bound obtained for $C(\sigma, n)$ takes the form of an integral. This integral can, for fixed dimension and given $\sigma$, be expressed in terms of certain hypergeometric functions. However, these expressions quickly become rather complicated. For the first few dimensions and some different values of $\sigma$ numerical values of the obtained upper and lower bounds for $C(\sigma, n)$ are displayed in Table~\ref{tbl:ValuesCcon1}.
\begin{table}[htp]
\settowidth{\MyLen}{0.0000\,/\,0.0000}
	\begin{tabular}{@{}|L|M{\MyLen} M{\MyLen} M{\MyLen} M{\MyLen} M{\MyLen}|@{}}
		\hline
		U / L & n=2 & n=3 & n=4 & n=5 & n=6\\
		\hline
		\vphantom{\int^1}\sigma=3/2 &  0.1334\, /\, 0.0846 & 0.0819\, /\, 0.0538 & 0.0572\, /\, 0.0391 & 0.0430\, /\, 0.0305 & 0.0339\, /\, 0.0247\\
		\vphantom{\int^1}\sigma=2 &  0.1228\, /\, 0.0808 & 0.0762\, /\, 0.0515 & 0.0537\, /\, 0.0375 & 0.0407\, /\, 0.0293 & 0.0323\, /\, 0.0239 \\
		\vphantom{\int^1}\sigma=5/2 & 0.1143\, /\, 0.0775 & 0.0716\, /\, 0.0495 & 0.0508\, /\, 0.0361 & 0.0387\, /\, 0.0283 & 0.0308\, /\, 0.0231 \\
		\vphantom{\int^1}\sigma=3 &  0.1074\, /\, 0.0747 & 0.0678\, /\, 0.0477 & 0.0484\, /\, 0.0349 & 0.0370\, /\, 0.0274 & 0.0296\, /\, 0.0224\\
		\hline
\end{tabular}
\vspace{2pt}
\captionof{table}{The obtained upper\,/\,lower bounds for $C(\sigma, n)$ for dimensions two through six and some different values of $\sigma$.}\label{tbl:ValuesCcon1}
\vspace{-10pt}
\end{table}

We proceed by giving the proof of \thmref{thm:ImprovedTracebound}, which in large follows along the same lines as the corresponding proof in \cite{LaptevGeisingerWeidl}.


\begin{proof}[Proof of \thmref{thm:ImprovedTracebound}]
	The first part of the theorem follows directly from that $\lambda_1(\Omega)\geq \frac{\pi^2}{4r^2}$, see~\cite{MR589137}. Therefore we may focus on the second case.

	Equation~{(13)} in \cite{LaptevGeisingerWeidl} states that for an open bounded set $\Omega\subset\R^n$, $\sigma \geq 3/2$ and $\Lambda>0$ we have that
	\begin{equation}\label{eq:13Laptev}
	 	\Tr(-\Delta_\Omega-\Lambda)^\sigma_- \leq \LT{\sigma}{n} \Lambda^{\sigma+n/2} \int_\Omega \int_{\S^{n-1}} \Bigl(1- \frac{1}{4 \Lambda d(x, u)^2}\Bigr)^{\!\sigma+n/2}_{\!+}\,d\nu(u)\,dx,
	\end{equation} 
	where $d\nu(u)$ is the normalized measure on the sphere. This inequality will be the starting point for the second part of the proof.

	Fix $x\in \Omega$ and choose $u_0\in \S^{n-1}$ such that $\delta(x)=d(x, u_0)$. Since everything is coordinate invariant we may assume that $u_0=(1, 0,\dots, 0)$ and let $\S^{n-1}_+ = \{u \in \S^{n-1} : \langle u, u_0\rangle >0 \}$. Denote by $a$ the intersection point of the ray $\{x+tu_0, t>0\}$ with $\partial \Omega$. Similarly, for $u\in \S^{n-1}_+$ let $b_u$ be the intersection point of the ray $\{x+tu, t>0 \}$ with the hyperplane through $a$ orthogonal to $u_0$, we note that this is nothing but the supporting hyperplane of $\Omega$ with normal $u_0$.

	We have that $d(x, u)\leq |x-b_u|$, and with $\theta_u$ denoting the angle between $u$ and $u_0$ we find that
	\begin{equation}
		d(x, u) \leq |x-b_u| =\frac{|x-a|}{\cos \theta_u} = \frac{\delta(x)}{\cos \theta_u}.
	\end{equation}
	Using the the antipodal symmetry of of $d(x, u)$ and inserting the above estimate into~\eqref{eq:13Laptev} one obtains that
	\begin{align}\label{eq:MeanOverSphere}
		\Tr(-\Delta_\Omega-\Lambda)^{\sigma}_- 
		&\leq 
		2 \LT{\sigma}{n} \Lambda^{\sigma + n/2}
		\int_\Omega \int_{\S_+^{n-1}} \Bigl(1- \frac{1}{4\Lambda d(x, u)^2}	\Bigr)_{\!+}^{\!\sigma+n/2}\,d\nu(u)\,dx\\
		&\leq	
		2\LT{\sigma}{n} \Lambda^{\sigma + n/2} \int_\Omega \int_{\S^{n-1}_+} \Bigl(1- \frac{\cos^2 \theta_u}{4\Lambda \delta(x)^2}\Bigr)_{\!+}^{\!\sigma+n/2}\,d\nu(u)\,dx.
	\end{align}
	We now switch to $n$-dimensional spherical coordinates such that $u_0$ is given by setting all angular coordinates to zero. Together with the rotational symmetry around $u_0$, this yields that
	\begin{equation}
		\Tr(-\Delta_\Omega-\Lambda)^{\sigma}_- 
		\leq 
		\LT{\sigma}{n} \Lambda^{\sigma + n/2}
		C_n \int_\Omega \int_0^{\pi/2} \Bigl(1- \frac{\cos^2 \theta}{4\Lambda d(x)^2}	\Bigr)_{\!+}^{\!\sigma+n/2}(\sin \theta)^{n-2}\,d\theta\,dx,
	\end{equation}
	where the normalization constant $C_n$ is given by
	\begin{equation}
	 	C_n = \Bigl(\int_{0}^{\pi/2}(\sin \theta)^{n-2}\,d\theta\Bigr)^{\!-1} = \frac{2\, \Gamma\bigl(\frac{n}{2}\bigr)}{\sqrt{\pi}\,\Gamma\bigl(\frac{n-1}{2}\bigr)}.
	\end{equation}

	We begin by rewriting the integral in \eqref{eq:MeanOverSphere} to easier obtain an expression of the desired form. 
	\begin{align}
	 		\Tr(-\Delta_\Omega-\Lambda)^{\sigma}_- 
	 	& \leq 
			\LT{\sigma}{n} \Lambda^{\sigma+n/2} C_n \int_\Omega  \int_{0}^{\pi/2} \Bigl(1- \frac{\cos^2 \theta}{4\Lambda \delta(x)^2}\Bigr)_{\!+}^{\!\sigma+n/2}(\sin \theta)^{n-2} \,d\theta\,dx\\[5pt]
		&=
		\LT{\sigma}{n} |\Omega| \Lambda^{\sigma+n/2}\\
		& \quad  -
			  \LT{\sigma}{n}  \Lambda^{\sigma+n/2}\int_\Omega \Bigl(1-C_n \int_{0}^{\pi/2} \Bigl(1- \frac{\cos^2 \theta}{4\Lambda \delta(x)^2}\Bigr)_{\!+}^{\!\sigma+n/2}(\sin \theta)^{n-2} \,d\theta\Bigr)dx\\
		&=
			\LT{\sigma}{n} |\Omega| \Lambda^{\sigma+n/2}\\
		& \quad -
			  \LT{\sigma}{n}  \Lambda^{\sigma+n/2}\int_{\R_+} |\partial \Omega_t| \Bigl(1-C_n \int_{0}^{\pi/2} \Bigl(1- \frac{\cos^2 \theta}{4\Lambda t^2}\Bigr)_{\!_+}^{\!\sigma+n/2}(\sin \theta)^{n-2} \,d\theta\Bigr)dt.
	\end{align}
	In the last step we make use of the coarea formula and that the distance function $\delta(x)$ satisfies the Eikonal equation $|\grad \delta|=1$ almost everywhere.

	By the definition of $C_n$ the expression in the outer integral is non-negative, that is 
	\begin{equation}
		1-C_n \int_{0}^{\pi/2} \Bigl(1- \frac{\cos^2 \theta}{4\Lambda t^2}\Bigr)_{\!+}^{\!\sigma+n/2}(\sin \theta)^{n-2} \,d\theta\geq 0.
	\end{equation}
	Therefore, using \thmref{thm:PerimeterBoundInnerParallel} one obtains that
	\begin{align}
		\Tr(-\Delta_\Omega&-\Lambda)^{\sigma}_- 
	 	 \leq
			\LT{\sigma}{n} |\Omega| \Lambda^{\sigma+n/2}\\
		 & -
			\LT{\sigma}{n} |\partial \Omega| \Lambda^{\sigma+n/2}\int_{\R_+} \Bigl(1- \frac{t}{r}\Bigr)^{\!n-1}_{\!+} \Bigl(1-C_n \int_{0}^{\pi/2} \Bigl(1- \frac{\cos^2 \theta}{4\Lambda t^2}\Bigr)_{\!+}^{\!\sigma+n/2}(\sin \theta)^{n-2} \,d\theta\Bigr)dt.
	\end{align}
	Letting $s=2\sqrt{\Lambda}\,t$ and using that $\Lambda \geq \frac{\pi^2}{4r^2}$ we find
	\begin{align}
		\int_{\R_+} \Bigl(1- \frac{t}{r} & \Bigr)^{\!n-1}_{\!+} \Bigl(1-C_n \int_{0}^{\pi/2} \Bigl(1- \frac{\cos^2 \theta}{4\Lambda t^2}\Bigr)_{\!+}^{\!\sigma+n/2}(\sin \theta)^{n-2} \,d\theta\Bigr)dt 
		\geq \\
	&
		\frac{1}{2\sqrt{\Lambda}}\int_{\R_+} \Bigl(1- \frac{s}{\pi}\Bigr)_{\!+}^{\!n-1} \Bigl(1-C_n \int_{0}^{\pi/2} \Bigl(1- \frac{\cos^2 \theta}{s^2}\Bigr)_{\!+}^{\!\sigma+n/2}(\sin \theta)^{n-2} \,d\theta\Bigr)ds.
	\end{align} 	
Since the integral above depends only on $n$ and $\sigma$ the claim follows with a lower bound on $C(\sigma, n)$ given by
\begin{equation}\label{eq:CconLowerBound}
	C(\sigma, n)\geq \frac{\LT{\sigma}{n}}{2 \, \LT{\sigma}{n-1}}I(\sigma, n),
\end{equation}
where
\begin{align}
	I(\sigma, n) &= \int_{\R_+} \Bigl(1- \frac{s}{\pi}\Bigr)_{\!+}^{\!n-1} \Bigl(1-C_n \int_{0}^{\pi/2} \Bigl(1- \frac{\cos^2 \theta}{s^2}\Bigr)_{\!+}^{\!\sigma+n/2}(\sin \theta)^{n-2} \,d\theta\Bigr)ds\\[5pt]
	&= 
	\int_0^\pi\Bigl(1- \frac{s}{\pi}\Bigr)^{\!n-1} \Bigl(1-C_n \int_{0}^{1} \Bigl(1- \frac{\varphi^2 }{s^2}\Bigr)_{\!+}^{\!\sigma+n/2} ( 1- \varphi^2)^{(n-3)/2} \,d\varphi\Bigr)ds.
\end{align}


To find upper estimates for the constants $C(\sigma, n)$ we argue as follows. For $\Lambda>\frac{\pi^2}{4r^2}$ our theorem says that
\begin{equation}
	\Tr(-\Delta_\Omega-\Lambda)^\sigma_- \leq \LT{\sigma}{n}|\Omega|\Lambda^{\sigma+n/2}-C(\sigma, n)\LT{\sigma}{n-1}|\partial\Omega|\Lambda^{\sigma+(n-1)/2}.
\end{equation}
But we know that the left-hand side is positive, and thus any positive zero of the polynomial on the right must be contained in the interval $(0, \frac{\pi^2}{4r^2}]$. Clearly the polynomial has exactly one positive zero $\Lambda_0$, given by
\begin{equation}
	\Lambda_0=\biggl(\frac{C(\sigma, n) \LT{\sigma}{n-1}|\partial\Omega|}{\LT{\sigma}{n}|\Omega|}\biggr)^{\!2}.
\end{equation}
Therefore we must have that
\begin{equation}
	\frac{\pi^2}{4r^2}\geq \biggl(\frac{C(\sigma, n) \LT{\sigma}{n-1}|\partial\Omega|}{\LT{\sigma}{n}|\Omega|}\biggr)^{\!2}.
\end{equation}
Rearranging the terms we find that for any convex domain $\Omega$ it should hold that
\begin{equation}\label{eq:Upper_Bound_C}
	C(\sigma, n) \leq \frac{\pi}{2r} \frac{\LT{\sigma}{n}|\Omega|}{\LT{\sigma}{n-1}|\partial\Omega|}.
\end{equation}

By using the coarea formula and Theorem~\ref{thm:PerimeterBoundInnerParallel} we find that
\begin{align}
	\frac{|\Omega|}{r|\partial\Omega|} 
	&= \frac{1}{r|\partial\Omega|}\int_0^r |\partial \Omega_t|\,dt 
	\geq 
	\frac{1}{r}\int_0^r \Bigl(1- \frac{t}{r}\Bigr)^{\!n-1}dt = \frac{1}{n},
\end{align}
where equality holds for a certain class of sets (see~\cite{Larson}). Inserting this into~\eqref{eq:Upper_Bound_C} we find that
\begin{equation}
	C(\sigma, n) \leq \frac{\sqrt{\pi}\,\Gamma(\sigma+ \frac{n+1}{2})}{4n\, \Gamma(\sigma+1+\frac{n}{2})},
\end{equation}
as this tends to zero when $\sigma$ or $n$ tends to infinity it is clear that the constant in \thmref{thm:ImprovedTracebound} must depend on both quantities.

Comparing the obtained upper and lower bounds we find that our proof provides a rather good estimate for $C(\sigma, n)$. This is also indicated by the numerical values in Table~\ref{tbl:ValuesCcon1}. 
\end{proof}


\section{Bounds on individual eigenvalues}\label{sec:IndividualEigenvalues}

Using the same methods as in \cite{LaptevGeisingerWeidl} we would like to obtain bounds for individual eigenvalues. However to analytically solve the equation that one obtains for $\Lambda$ is no simple task, since it involves solving an $n$-th order polynomial equation. It is, however, not difficult to numerically compute lower bounds. Nonetheless, we are able to conclude that the bounds implicitly given by our improved trace bounds in fact improve those given by the Li--Yau inequality for a certain range of $k$ (which, in a rather complicated way, depends on $n$). As an introduction to what is to come, we state and prove the following result for the two-dimensional case. The proof is precisely the same as that given in~\cite{LaptevGeisingerWeidl}. 

\begin{corollary}[\cite{LaptevGeisingerWeidl}, Corollary~5.2]\label{Cor:IndBoundsTwoDim}
	Let $\Omega\subset\R^n$ be a bounded, convex domain. Then with $C=C(3/2, 2)$ given by \thmref{thm:ImprovedTracebound} we for any $k\in \N$ and $\alpha \in (0,1)$ have that
	\begin{align}
		\frac{\lambda_k(\Omega)}{1-\alpha}
		&\geq
		10 \pi \alpha^{3/2}\frac{k}{|\Omega|} +
			\frac{15 \pi C}{8} \frac{|\partial\Omega|}{|\Omega|}
			\sqrt{10 \pi \alpha^{3/2} \frac{k}{|\Omega|} +
			\frac{225 \pi^2 C^2}{256} \frac{|\partial\Omega|^2}{|\Omega|^2}}\\
		&\quad
			+\frac{225 \pi^2 C^2}{128} \frac{|\partial\Omega|^2}{|\Omega|^2}.
	\end{align}
\end{corollary}

\begin{proof}
	We let $N(\Lambda)=\Tr(-\Delta_\Omega-\Lambda)^0_-$ be the counting function of eigenvalues less than $\Lambda$. For $\sigma>0$ and all $\Lambda>0, \tau > 0$ it is shown in \cite{Laptev2} that
	\begin{equation}\label{eq:TraceIneq0sigma}
		N(\Lambda)\leq (\tau\Lambda)^{-\sigma}\Tr(-\Delta_\Omega-(1+\tau)\Lambda)^\sigma_-.
	\end{equation}
	Applying this with $\sigma=3/2$, we can use \thmref{thm:ImprovedTracebound} and for $\Lambda\geq \frac{\pi^2}{4 r^2}$ estimate
	\begin{align}
		N(\Lambda)&\leq (\tau\Lambda)^{-3/2}\bigl(L_{3/2,2}^{\textrm{cl}}|\Omega| ((1+\tau)\Lambda)^{5/2}-C(3/2, 2) L_{3/2,1}^{\textrm{cl}}|\partial \Omega| ((1+\tau)\Lambda)^{2}\bigr)\\
		&=
		\LT{3/2}{2}|\Omega|\frac{(1+\tau)^{5/2}}{\tau^{3/2}}\Lambda- C(3/2, 2)\LT{3/2}{1}|\partial\Omega|\frac{(1+\tau)^2}{\tau^{3/2}}\sqrt{\Lambda}.
	\end{align}
	Substituting $\tau = \alpha /(1-\alpha)$ for $\alpha\in (0, 1)$ and using that $N(\lambda_k)\geq k$ we find that
	\begin{align}
		k \leq \LT{3/2}{2}\alpha^{-3/2}|\Omega|\frac{\lambda_k(\Omega)}{1-\alpha}- C(3/2, 2)\LT{3/2}{1}|\partial\Omega|\alpha^{-3/2}\sqrt{\frac{\lambda_k(\Omega)}{1-\alpha}}.
	\end{align}
	Since the right-hand side is a convex quadratic polynomial in $\sqrt{\frac{\lambda_k(\Omega)}{1-\alpha}}$ which vanishes at zero, there is exactly one positive solution to where this polynomial is equal to $k$. By monotonicity this solution provides a lower bound for $\sqrt{\frac{\lambda_k(\Omega)}{1-\alpha}}$. Through some algebraic manipulations this yields that
	\begin{align}
		\frac{\lambda_k(\Omega)}{1-\alpha} 
		&\geq 
		\Biggl(\frac{C(3/2, 2)\LT{3/2}{1}|\partial\Omega|+\bigl((C(3/2, 2)\LT{3/2}{1}|\partial\Omega|)^2+4 k \LT{3/2}{2}|\Omega|\alpha^{3/2}\bigr)^{1/2}}{2\LT{3/2}{2}|\Omega|}\Biggr)^2\\
		&=
		 \frac{\alpha^{3/2} k}{L_{3/2, 2}^{\textrm{cl}} |\Omega|} +
			C(3/2, 2)\frac{ L_{3/2, 1}^{\textrm{cl}}}{L_{3/2, 2}^{\textrm{cl}}} \frac{|\partial\Omega|}{|\Omega|}
			\sqrt{\frac{\alpha^{3/2} k}{L_{3/2, 2}^{\textrm{cl}} |\Omega|} +
			\frac{C(3/2, 2)^2}{4} \biggl(\ptS\frac{L^{\textrm{cl}}_{3/2, 1}}{L^{\textrm{cl}}_{3/2, 2}}\ptS\biggr)^{\!2} \frac{|\partial\Omega|^2}{|\Omega|^2}}\\
		&\quad
			+\frac{C(3/2, 2)^2}{2} \biggl(\ptS\frac{L_{3/2, 1}^{\textrm{cl}}}{L_{3/2, 2}^{\textrm{cl}}}\ptS\biggr)^{\!2} \frac{|\partial\Omega|^2}{|\Omega|^2}.
	\end{align}
	Inserting the values of the Lieb--Thirring constants we obtain the desired expression.
\end{proof}

For dimensions $3$ and $4$ the same method can be used to get explicit bounds for $\lambda_k(\Omega)$, but the expressions obtained become rather intractable as they involve the formula for a root of a third respectively fourth degree polynomial. However, since we already know that the bounds given by the Li--Yau inequality are better for large $k$ a problem of interest is to find in what range of $k$ our bound is an improvement of that given by Li--Yau~\eqref{eq:LiYau}. In this direction the estimates in Corollary~5.2 of \cite{LaptevGeisingerWeidl} improve the Li--Yau inequality for $n=2$ when $k\leq 23$. With the new improved estimates for $C(3/2, 2)$ obtained here this is increased to all $k< 40$.

Let $B(\Omega, k, n)$ be such that $\lambda_k(\Omega)\geq B(\Omega, k, n)$ is the bound implied by \thmref{thm:ImprovedTracebound}. In what follows we will provide, for general $n$, a lower bound for $k^*$, which is such that for any integer $k < k^*$ we have that
\begin{equation}
	\lambda_k(\Omega)\geq B(\Omega, k, n) > \Gamma\Bigl(\frac{n}{2}+1\Bigr)^{\!2/n}\frac{4 \pi n}{n+2}\Bigl(\ptS\frac{k}{|\Omega|}\ptS\Bigr)^{\!2/n},
\end{equation}
where the right-hand side is the Li--Yau inequality. As in~\cite{LaptevGeisingerWeidl} we also consider the question for which $\lambda_k$, with $k>2$, the bound is an improvement of that implied by the Krahn--Szeg\H o inequality:
\begin{equation}\label{eq:KrahnSzegoIneq}
	\lambda_k(\Omega)\geq\lambda_2(\Omega) \geq \pi \Gamma\Bigl(\frac{n}{2}+1\Bigr)^{\!-2/n} \Bigl(\ptS\frac{2}{|\Omega|}\ptS\Bigr)^{\!2/n} j_{n/2-1, 1}^2,
\end{equation}
where $j_{m, 1}$ denotes the first positive zero of the Bessel function $J_{m}$.

\begin{theorem}\label{thm:BoundOnKstar}
	Let $\Omega\subset \R^n$ be a bounded, convex domain. Then, there exist $k_*, k^* > 0$ depending only on the dimension such that for all $k$ satisfying $k_*\!< k< k^*$ the lower bound
	\begin{equation}
		\lambda_k(\Omega)\geq B(\Omega, k, n)
	\end{equation}
	is an improvement of the both the Li--Yau inequality and the bound in~\eqref{eq:KrahnSzegoIneq}.
	Moreover, for $n$ sufficiently small the set of such $k$ is non-empty and we have that
	\begin{align}
		k^*  &\geq
		\frac{3^n}{2^n} \frac{\pi^{n} n^n}{\Gamma\bigl(\frac{n}{2}+1\bigr)^{\!2}} \biggl(
		\frac{
			C(3/2, n) (n+2)^{1/2}(n+3)^{2+n/2} \Gamma(n+2)
		}{
			3 \cdot 2^n n (n+3)^{(n+3)/2} \Gamma\bigl(\frac{n}{2}+2\bigr)\Gamma\bigl(\frac{n}{2}\bigr) 
			-
			3^{3/2}(n+2)^{n/2}  \Gamma(n+4)
		}
		\biggr)^{\!n},\\[4pt]
		k_* & \leq
			\Bigl(\frac{n+2}{n}\Bigr)^{\!n/2} \frac{2^{1-n}}{\Gamma\bigl(\frac{n}{2}+1\bigr)^{\!2}}\, j_{n/2-1, 1}^n.
	\end{align}
	In particular, for the first few dimensions the obtained bounds are displayed in the Table~\ref{tbl:boundsOnK} below.
\end{theorem}

\begin{table}[H]
	\vspace{5pt}
	\settowidth{\MyLen}{999}
	\begin{tabular}{@{}|R|M{\MyLen}M{\MyLen}M{\MyLen}M{\MyLen}M{\MyLen}M{\MyLen}M{\MyLen}|@{}}
		\hline
		n= & 2 & 3 & 4 & 5 & 6 & 7 & 8\\
		\hline
		\vphantom{\int^1}k^*\geq & 40 & 91 & 165 & 255 & 332 & 392 & 412\\
		\vphantom{\int^1}k_* \leq & 6 & 10 & 16 & 25 & 38 & 59 & 91\\
		\hline
	\end{tabular}
	\vspace{2pt}
	\captionof{table}{The upper respectively lower bounds for $k_*, k^*$.}\label{tbl:boundsOnK}
\end{table}
\vspace{-17pt}
As is indicated by the table above the gap between $k^*$ and $k_*$ has a maximum around dimension $n=7$, after which the gap appears to close rather quickly. Using the obtained upper bounds for $C(3/2, n)$ it is not difficult to show that $k^*$ will tend to zero as $n\to \infty$.


\begin{proof}[Proof of \thmref{thm:BoundOnKstar}]
By the Li--Yau inequality we know that for an open set $\Omega\subset \R^n$ we have that
\begin{equation}
	\lambda_k(\Omega) \geq  \Gamma\Bigl(\frac{n}{2}+1\Bigr)^{\!2/n}\frac{4\pi n}{n+2}\Bigl(\ptS\frac{k}{|\Omega|}\ptS\Bigr)^{\!2/n}.
\end{equation}
Solving this for $k$ we find that it is equivalent to the bound
\begin{equation}
	k\leq \Bigl(\frac{n+2}{4\pi n}\Bigr)^{\!n/2} \frac{|\Omega|}{\Gamma\bigl(\frac{n}{2}+1\bigr)} \lambda_k^{n/2}.
\end{equation}
By monotonicity this implies that
\begin{equation}\label{eq:CountingBoundLiYau}
 	N(\Lambda) \leq \Bigl( \frac{n+2}{4\pi n}\Bigr)^{\!n/2} \frac{|\Omega|}{\Gamma\bigl(\frac{n}{2}+1\bigr)} \Lambda^{n/2}=:P_{LY}(\Lambda).
\end{equation}

Using \eqref{eq:TraceIneq0sigma} we conclude from Theorem~\ref{thm:ImprovedTracebound} that if $\Lambda>0$ and $\tau>0$ then
\begin{equation}
	N(\Lambda) \leq \Bigl(\LT{3/2}{n} |\Omega| \frac{(1+\tau)^{ (n+3)/2 }}{ \tau^{3/2}} \Lambda^{n/2} - C(3/2, n) \LT{3/2}{n-1} |\partial \Omega| \frac{(1+\tau)^{1+ n/2}}{\tau^{3/2}} \Lambda^{(n-1)/2}\Bigr)_{\!+} =: P(\Lambda).
\end{equation}
It is clear that both $P(\Lambda)$ and $P_{LY}(\Lambda)$ are continuous and monotonically increasing for $\Lambda\geq 0$. By monotonicity the bound given by $\lambda_k \geq B(\Omega, k, n) = P^{-1}(k)$ is sharper than the Li--Yau inequality precisely when $P(\Lambda)< P_{LY}(\Lambda)$. 

Further, we have that
\begin{equation}
	\LT{3/2}{n} \frac{(1+\tau)^{ (n+3)/2 }}{ \tau^{3/2}} \geq \LT{3/2}{n} \frac{(n+3)^{(n+3)/2}}{3^{3/2} n^{n/2}} > \Bigl( \frac{n+2}{4\pi n}\Bigr)^{\!n/2} \Gamma\Bigl(\frac{n}{2}+1\Bigr)^{\!-1},
\end{equation}
where we used that the left-hand side is minimal (for $\tau>0)$ when $\tau=n/3$. And thus the polynomial $P_{LY}$ is asymptotically larger than $P$. Hence it is clear that there exists a unique $\Lambda^*>0$ such that $P(\Lambda^*)=P_{LY}(\Lambda^*)$. Correspondingly, if we let $k^*$ be such that for all $k< k^*$ we have
\begin{equation}
	P^{-1}(k) > P_{LY}^{-1}(k) = \Gamma\Bigl(\frac{n}{2}+1\Bigr)^{\!2/n}\frac{4\pi n}{n+2} \Bigl(\ptS\frac{k}{|\Omega|}\ptS\Bigr)^{\!2/n},
\end{equation}
then $k^*$ is the smallest integer larger than $P(\Lambda^*)=P_{LY}(\Lambda^*)$. By monotonicity finding a lower bound for $\Lambda^*$ corresponds to finding a lower bound for $k^*$. 

We proceed by calculating $\Lambda^*=\Lambda^*(n, \Omega, \tau)$. After equating the two polynomials, a simple calculation using that $\Lambda^*>0$ gives us the solution
\begin{align}
	\Lambda^* &= 
	\biggl(\frac{
		C(3/2, n) |\partial \Omega|\Gamma\bigl(\frac{n}{2}+1\bigr) \LT{3/2}{n-1}  \frac{(1+\tau)^{1+n/2}}{\tau^{3/2}}
	}
	{
		|\Omega|\Gamma\bigl(\frac{n}{2}+1\bigr) \LT{3/2}{n}  \frac{(1+\tau)^{(n+3)/2} }{ \tau^{3/2}}   -  
		|\Omega|\bigl( \frac{n+2}{4\pi n}\bigr)^{\!n/2}
	}\biggr)^{\!2} \\
	&=
	\biggl(\frac{
		C(3/2, n)\, \Gamma\bigl(\frac{n}{2}+1\bigr)  \LT{3/2}{n-1} (1+\tau)^{1+n/2}
	}
	{
		\Gamma\bigl(\frac{n}{2}+1\bigr) \LT{3/2}{n} (1+\tau)^{ (n+3)/2 }   -  
		\bigl( \frac{n+2}{4\pi n} \bigr)^{\!n/2} \tau^{3/2}
	}\biggr)^{\!2}
	\frac{|\partial \Omega|^2}{|\Omega|^2}.
\end{align}

We can now insert this expression into either of our two polynomials to attempt to estimate $k^*$. Since $P_{LY}$ is a monomial it makes our computations slightly simpler.
\begin{align}
	P_{LY}(\Lambda^*) 
	&= \Bigl( \frac{n+2}{4\pi n}\Bigr)^{\!n/2} \frac{|\Omega|}{\Gamma\bigl(\frac{n}{2}+1\bigr)} (\Lambda^*)^{n/2}\\
	&=
	\Bigl( \frac{n+2}{4\pi n}\Bigr)^{\!n/2} \frac{1}{\Gamma\bigl(\frac{n}{2}+1\bigr)} \biggl(\frac{
		C(3/2, n)\,\Gamma\bigl(\frac{n}{2}+1\bigr) \LT{3/2}{n-1} (1+\tau)^{1+n/2}
	}
	{
		\Gamma\bigl(\frac{n}{2}+1\bigr) \LT{3/2}{n} (1+\tau)^{ (n+3)/2}   -  
		\bigl( \frac{n+2}{4\pi n} \bigr)^{\!n/2} \tau^{3/2}
	}\biggr)^{\!n}
	\frac{|\partial \Omega|^n}{|\Omega|^{n-1}}.
\end{align}
Note that $P_{LY}(\Lambda^*)$ behaves very nicely with respect to both the isoperimetric ratio of our domain and the constant $C(3/2, n)$.

Now as this expression is rather messy, especially in its dimensional dependence, it is not the easiest task to calculate its integer part. Even trying to optimize this in $\tau$ is a rather intricate problem. But considering where $\tau$ comes from in our argument, and that the bound holds for any $\tau>0$, we simply choose the $\tau$ which minimizes the leading coefficient of $P(\Lambda)$. A simple calculation shows that this is attained at $\tau=3/n$. Inserting this into the expression above we lose the dependence of $\tau$ and obtain that
\begin{equation}
	k^*  \geq
		\frac{3^n}{2^n} \frac{\pi^{n/2}}{\Gamma\bigl( \frac{n}{2}+1\bigr)}
			\biggl(
			\frac{
				C(3/2, n) (n+2)^{1/2} (n+3)^{2+n/2} \Gamma(n+2)
			}{
					3\cdot 2^n n (n+3)^{(n+3)/2} \Gamma(\frac{n}{2}+2)\Gamma(\frac{n}{2})-3^{3/2} (n+2)^{n/2} \Gamma(n+4)
			}
			\biggr)^{\!n}
			\frac{|\partial \Omega|^n}{|\Omega|^{n-1}}.
\end{equation} 

Using the isoperimetric inequality we may further bound this, and thus also lose the domain dependence. This gives the following bound which now depends only on the dimension
\begin{equation}
	k^*  \geq
		\frac{3^n}{2^n} \frac{\pi^{n} n^n}{\Gamma\bigl(\frac{n}{2}+1\bigr)^{\!2}} \biggl(
		\frac{
			C(3/2, n)(n+2)^{1/2}(n+3)^{2+n/2} \Gamma(n+2)
		}{
			3 \cdot 2^n n (n+3)^{(n+3)/2} \Gamma\bigl(\frac{n}{2}+2\bigr)\Gamma\bigl(\frac{n}{2}\bigr) 
			-
			3^{3/2}(n+2)^{n/2}  \Gamma(n+4)
		}
	\biggr)^{\!n}.
\end{equation}

As in~\cite{LaptevGeisingerWeidl} we can supplement these bounds from below. Let $\Lambda_{KZ}$ denote the bound for $\lambda_2(\Omega)$ given by~\eqref{eq:KrahnSzegoIneq}, that is
\begin{equation}
	\Lambda_{KZ}=\pi \Gamma\Bigl(\frac{n}{2}+1\Bigr)^{\!-2/n} \Bigl(\ptS\frac{2}{|\Omega|}\ptS\Bigr)^{\!2/n} j_{n/2-1, 1}^2,
\end{equation}
where again $j_{m, 1}$ denotes the first positive zero of the Bessel function $J_m$. By the same reasoning as before we can conclude that $k_* \leq P(\Lambda_{KZ})$. If $\Lambda_{KZ}\leq \Lambda^*$ we have that $P(\Lambda_{KZ})\leq P_{LY}(\Lambda_{KZ})$ thus if this is the case $P_{LY}(\Lambda_{KZ})$ is an upper bound for $k_*$. Moreover, if $\Lambda_{KZ}>\Lambda^*$ then $k_*\geq k^*$ and the range of $k$ where our implicit bounds improve the Li--Yau bound and that implied by Krahn--Szeg\H o is empty, and therefore we can restrict our interest to the first case. Calculating we find that
\begin{align}
 	k_* \leq P_{LY}(\Lambda_{KZ})  = \Bigl(\frac{n+2}{n}\Bigr)^{\!n/2} \frac{2^{1-n}}{\Gamma\bigl(\frac{n}{2}+1\bigr)^2} j_{n/2-1, 1}^n,
\end{align}
which completes the proof.\end{proof}

\vspace{3pt}\noindent{\bf Acknowledgements.} It is a pleasure to thank Ari Laptev for suggesting the problem studied here and for his help and encouragement. The author would also like to thank Aron Wennman and Eric Larsson for many discussions concerning the results in this paper. Finally we thank the referee for their thorough reading and many helpful suggestions. The author is supported by the Swedish Research Council grant no.\ 2012-3864.


\providecommand{\noop}[1]{}\def\cprime{$'$}
\providecommand{\bysame}{\leavevmode\hbox to3em{\hrulefill}\thinspace}
\providecommand{\MR}{\relax\ifhmode\unskip\space\fi MR }
\providecommand{\MRhref}[2]{%
  \href{http://www.ams.org/mathscinet-getitem?mr=#1}{#2}
}
\providecommand{\href}[2]{#2}

\end{document}